\documentclass{article}
\usepackage{arxiv}

\usepackage{amsmath}
\usepackage{amssymb}
\usepackage{amsthm}
\usepackage{xcolor}

\usepackage{fontspec}
\usepackage{microtype}
\usepackage{xspace}
\usepackage{fontawesome5} % external-link icon used in \docref
\usepackage[pagebackref]{hyperref}

\usepackage{listings}
\definecolor{keywordcolor}{rgb}{0.7, 0.1, 0.1}   % red
\definecolor{tacticcolor}{rgb}{0.0, 0.1, 0.6}    % blue
\definecolor{commentcolor}{rgb}{0.4, 0.4, 0.4}   % grey
\definecolor{symbolcolor}{rgb}{0.0, 0.1, 0.6}    % blue
\definecolor{sortcolor}{rgb}{0.1, 0.5, 0.1}      % green
\definecolor{attributecolor}{rgb}{0.7, 0.1, 0.1} % red
\definecolor{mycolor}{RGB}{0,100,200}            % frame color for listings
\definecolor{mycolorSubtle}{RGB}{245,250,255}    % background for \myinline

\usepackage{xpatch}
\usepackage{realboxes}
\DeclareRobustCommand{\myinline}{\lstinline}
\makeatletter
\xpretocmd\myinline{\Colorbox{mycolorSubtle}\bgroup\appto\lst@DeInit{\egroup}}{}{}
\makeatother

\newtheorem{theorem}{Theorem}[section]
\newtheorem{lemma}[theorem]{Lemma}
\newtheorem{corollary}[theorem]{Corollary}

\theoremstyle{definition}
\newtheorem{definition}[theorem]{Definition}

\newtheorem{remark}[theorem]{Remark}
\newenvironment{proofsketch}{\begin{proof}[Proof sketch]}{\end{proof}}

\newcommand{\Z}{{\mathbb{Z}}}

\DeclareMathOperator{\val}{val}

\DeclareMathOperator{\indeg}{indeg}
\DeclareMathOperator{\Div}{Div}

\newcommand{\Mathlib}{\textsc{Mathlib}\xspace}

\newcommand{\extlinkicon}{%
  \textsuperscript{\fontsize{6pt}{6pt}\selectfont\faIcon{external-link-alt}}%
}
\DeclareRobustCommand{\docref}[2]{%
  \href{https://dhyeymavani.com/chip-firing-with-lean/docs/ChipFiringWithLean/#1.html\##2}{\texttt{\detokenize{#2}}\,\extlinkicon}%
}

\title{Formalizing Chip-Firing and Riemann--Roch for Graphs in Lean~4}
\shorttitle{Formalizing chip-firing and Riemann--Roch}

\author{
 Dhyey Dharmendrakumar Mavani \\
  Amherst College \\
  Amherst, MA 01002 \\
  \texttt{ddmavani2003@gmail.com} \\
   \And
 Nathan Pflueger \\
  Amherst College \\
  Amherst, MA 01002 \\
  \texttt{npflueger@amherst.edu} \\
}

\begin{document}
\maketitle

\begin{abstract}
The Riemann--Roch theorem for graphs, due to Baker and Norine, is a foundational result establishing a powerful analogy between finite graphs and algebraic curves. We describe a complete formal proof of this theorem implemented in the Lean~4 theorem prover. Our formalization includes the existence and uniqueness of $q$-reduced divisors, a modified form of Dhar's burning algorithm, the bijection between acyclic orientations with unique source and maximal superstable configurations, and Clifford's theorem. We also include several challenges for future formalization.
\end{abstract}


%% ============================================================
\section{Introduction}
\label{sec:intro}
%% ============================================================
A chip-firing game is a discrete dynamical system on a graph in which vertices redistribute integer-valued tokens among their neighbors according to simple local rules: a vertex ``fires'' by sending one chip along each incident edge to its neighbors. Chip-firing games exhibit rich combinatorial structure and have been the subject of extensive study. A survey of this theory, whose exposition and notation we mainly follow in the present paper, is Corry and Perkinson's textbook~\cite{corry2018divisors}. Indeed we hope our repository may be a helpful ``Lean companion'' for a reader studying the first five chapters of that book.

There are various ways to create a literal game from these rules, the simplest of which is \emph{the dollar game}: given an allocation of chips on the graph, the player wins if they can perform a sequence of chip-firing moves after which no vertex is in debt.
For an elementary introduction to the dollar game, see \cite{lambogliaulirsch2021}, or try the interactive online game at \url{https://thedollargame.io}~\cite{dollargame}.

Chip-firing games are the setting of a deep analogy between finite graphs and Riemann surfaces. Roughly speaking, the class of a chip placement modulo chip-firing moves is analogous to a line bundle on a compact Riemann surface. The strength of this analogy was demonstrated by Baker and Norine~\cite{baker2007riemann} with their graph-theoretic Riemann--Roch theorem, which is a purely combinatorial avatar of the classical Riemann--Roch theorem for Riemann surfaces.
After this watershed moment, the analogy between graphs and Riemann surfaces has borne fruit in both directions. For example, the tools in \cite{baker2008} make it possible to prove theorems about Riemann surfaces using combinatorics of chip-firing games. This technique found early success in a tropical proof, via chip-firing games, of the Brill--Noether theorem~\cite{cools2012tropical}. Since then, the analogy has been used in many cases to prove new theorems about Riemann surfaces; one example is the first proof of the fixed-gonality analogue of the Brill--Noether theorem, which was conjectured by the second author~\cite{pflueger2017brill} and proved by Jensen and Ranganathan~\cite{jensenranganathan2021fixedgonality}, in both cases by way of chip-firing games. More about the analogy to Riemann surfaces, and the rich interplay it has made possible, may be found in the popular article \cite{hartnett2018tinkertoy}, surveys \cite{jensen2021chipfiringcurves,ranganathan2023tropicalforwards}, and the references therein.

Our objective is to lay the groundwork for the formalization of the theory of chip-firing games, especially the aspects relevant to Riemann surfaces.
The first major milestone in this effort is a formalization of the graph-theoretic Riemann--Roch theorem in Lean~4. The formalization contains no \texttt{sorry}s, introduces no additional axioms, and relies on the \Mathlib library~\cite{mathlib2020}.
Our code is available at the link below, and the reader will find links to its documentation distributed throughout the text of this paper.

\begin{center}\url{https://github.com/DhyeyMavani2003/chip-firing-with-lean}\end{center}

To our knowledge, this is the first formalized proof of Riemann--Roch for graphs in any proof assistant. Our work includes new Lean types for divisors, configurations, and graph orientations, together with a development of the Baker--Norine rank function. Along the way, we prove the existence and uniqueness of $q$-reduced divisors, formalize a version of Dhar's burning algorithm, establish the bijection between acyclic orientations with unique source and maximal superstable configurations, and deduce Clifford's theorem. The formalization closely follows the exposition \cite{corry2018divisors} and was developed as part of an honors thesis~\cite{mavani2025chipfiringthesis} at Amherst College. The project comprises approximately 5,400 lines of Lean~4 code across eight files.
At the end of \myinline|RiemannRoch.lean| we also include formalized statements (without proofs) of four conjectures of Matt Baker, two of which are now proved, and two of which are open, which we offer as challenges to the formalization community and possible candidates to benchmark artificial intelligence tools.

The remainder of this paper is organized as follows. Section~\ref{sec:notation} introduces the terminology of divisors, configurations, and Baker--Norine rank. With the notation in place, we then explain the proof backwards. 
Section~\ref{sec:rr} presents the formalization of the Riemann--Roch theorem itself, and reduces its proof to a duality of maximal unwinnable divisors.
Section~\ref{sec:qred} discusses a key technical step: the existence of $q$-reduced divisors and superstable configurations.
Section~\ref{sec:dhar} explains Dhar's burning algorithm in its usual form and how we have chosen to formalize it.
Section~\ref{sec:orient} describes the formalization of graph orientations and the orientation--superstable bijection, and completes the proof of Riemann--Roch. 
Section~\ref{sec:future} concludes with possibilities for future work, including a statement of four formalization challenges based on conjectures of Matt Baker.

\subsection*{Acknowledgements}
We are grateful to Matt Baker and Daniel Velleman for helpful comments.
The second author was supported by a D. Crary Sabbatical Fellowship from Amherst College, and is grateful to the Max Planck Institute for Mathematics in the Sciences for its hospitality during the conclusion of this work.

\subsection*{Disclosure on generative artificial intelligence}

Portions of our code repository were developed with the aid of generative artificial intelligence, including GitHub Copilot, Claude Code, and OpenAI Codex. The primary input from these tools was autocompletion of routine steps, simplification of code, and initial drafting of docstrings. All formal statements, organizational choices, and design decisions were made by the authors. Generative AI was also used for limited proofreading and revisions of this manuscript.

%% ============================================================
\section{Terminology}
\label{sec:notation}
%% ============================================================

This section recalls the necessary definitions, following the exposition of Corry and Perkinson~\cite{corry2018divisors} and Baker and Norine~\cite{baker2007riemann}.
Here and throughout the paper, we include links to our repository's documentation, which are indicated with the \extlinkicon\, symbol.
From there the reader can find the statement in Lean syntax and a link to the source code.

\subsection{Graphs}
Throughout the paper, $G = (V, E)$ denotes a finite, loopless multigraph: $V(G)$ is a finite nonempty set of vertices, and $E(G)$ is a multiset of unordered pairs $\{v, w\}$ with $v \ne w$. This is formalized in the \docref{Basic}{CFGraph} structure. We will usually assume that $G$ is connected, which is formalized as \docref{Basic}{graph_connected}. For this project, the most practical definition of connectivity is in terms of cut sets: if $S$ is a nonempty proper subset of $V(G)$, then there is at least one edge from $S$ to its complement.

\begin{remark}\label{rem:mathlibgraph}
In future updates, we plan to refactor our definition to build on \Mathlib's existing multigraph definition,
\href{https://leanprover-community.github.io/mathlib4_docs/Mathlib/Combinatorics/Graph/Basic.html#Graph}{\texttt{\detokenize{Graph}}\,\extlinkicon}%
. The reason for our bespoke \docref{Basic}{CFGraph} definition is simple: when this project began in 2024, this definition was not yet available in \Mathlib. At the time of writing, the \Mathlib graph library has pull requests under review that may simplify this refactor, so we have chosen to wait until that process is completed before proceeding.
\end{remark}

For practical purposes, almost all interactions with $G$ are mediated through its adjacency matrix, which is represented by a symmetric, nonnegative function \docref{Basic}{num_edges} with zero diagonal. Users of the library are strongly urged to treat this function as the outward-facing interface of $G$, so that the internal representation is not exposed unnecessarily.

The \emph{genus} (\docref{Basic}{genus}) of $G$ is the integer $g = |E(G)| - |V(G)| + 1$. For connected graphs this is the first Betti number, also called the cyclomatic number; it is called the genus in chip-firing games to highlight the analogy with Riemann surfaces.

\subsection{Divisors}
A \emph{divisor} (\docref{Basic}{CFDiv}) on $G$ is an element of the free abelian group $\Div(G) = \Z^V$, or equivalently a function $V(G) \to \Z$. The abelian group structure is inherited automatically from \Mathlib's
\href{https://leanprover-community.github.io/mathlib4_docs/Mathlib/Algebra/Group/Pi/Basic.html#Pi.addCommGroup}{\texttt{\detokenize{Pi.addCommGroup}}\,\extlinkicon}.
If $D \in \Div(G)$, then $D(v)$ gives the number of chips at vertex $v$, with negative values representing debt. The \emph{degree} (\docref{Basic}{deg}) of a divisor is $\deg(D) = \sum_{v \in V} D(v)$, which in the dollar game interpretation represents the total wealth in the system.

Given a divisor $D$ and a vertex $v$, a \emph{firing move} at $v$ produces a new divisor $D'$ in which $v$ sends one chip along each incident edge, represented by adding the \docref{Basic}{firing_vector} of $v$. The \docref{Basic}{principal_divisors} are the subgroup generated by these firing vectors. Two divisors are \emph{linearly equivalent} (\docref{Basic}{linear_equiv}) if they differ by a principal divisor, meaning one can be obtained from the other by a sequence of firing moves. We write $D \sim D'$ to indicate that $D$ and $D'$ are linearly equivalent.

An important special case is the principal divisor resulting from firing all vertices in a set $S$ exactly once. Such an operation is called a \emph{set firing}. The effect of set-firing $S$ is to send one chip along each edge from $S$ to $V(G) \setminus S$. Any principal divisor may be represented as a sequence of set firings from nested sets $S_n \subseteq S_{n-1} \subseteq \ldots \subseteq S_1$, where $S_i$ is the set of vertices that are fired at least $i$ times.

A divisor is \docref{Basic}{effective} if it assigns nonnegative values to every vertex. A divisor is \docref{Basic}{winnable} if it is linearly equivalent to some effective divisor. The objective of the dollar game is thus to determine whether a given divisor is winnable.

One specific divisor is of particular importance: the \docref{Orientation}{canonical_divisor}, usually denoted $K$. It is defined by $K(v) = \val(v) - 2$, where $\val(v)$ is the degree of vertex $v$. A straightforward calculation via the handshaking lemma shows that $\deg K = 2g-2$. Though straightforward in prose, the argument is a bit technical in Lean; it can be found at \docref{Orientation}{degree_of_canonical_divisor}. The canonical divisor of a graph, just like a canonical divisor on a compact Riemann surface, plays a crucial symmetrizing role in the Riemann--Roch theorem.

\subsection{Configurations and superstability}

We follow the terminology of \cite[\S 2]{corry2018divisors} in making the following subtle linguistic distinction: a \emph{configuration} (\docref{Config}{Config}) with respect to a marked vertex $q$ is a divisor in which we forget the number of chips at $q$. In other words, the group of configurations is $\operatorname{Div}(G) / q \Z$. One imagines $q$ as a sink into which chips can be fired but where they subsequently disappear.
However, our code modifies this usage in one important respect: our \myinline|Config G q| type includes only \emph{nonnegative} configurations, i.e. those with no debt at any $v \ne q$. This is a matter of convenience: most of the load-bearing theorems about configurations in our repository assume nonnegativity anyway. We also make a pragmatic implementation choice and represent a configuration as a divisor in which $D(q) = 0$, even though we really want to think of $D(q)$ as undefined. This choice allows us to invoke chip-firing operations on configurations in a simple way; we simply re-zero the number of chips at $q$ afterward. This also allows us a mild abuse of notation in prose (which must be mediated carefully by the conversion functions \docref{Config}{toConfig} and \docref{Config}{toDiv} in Lean): if $c$ is a configuration with respect to a vertex $q$, we will write ``the divisor $c-q$'' for the divisor $D$ with $D(v) = c(v)$ at $v \ne q$ and $D(q) = -1$. This is consistent with the notation of \cite{corry2018divisors}.

Of particular importance are the \docref{Config}{superstable} configurations. These are the configurations for which there is no debt, but from which firing any nonempty set \emph{not including $q$} would put some vertex of that set into debt. The corresponding notion for divisors is
\docref{Basic}{q_reduced}. A key feature of configurations, which occupies a significant fraction of the lines of code in our repository, is that any configuration can be made superstable by firing a finite sequence of sets not containing $q$, and that the resulting superstable configuration is unique. We return to discuss the formalization of this result in Section~\ref{sec:qred}.

\begin{remark}
The presence of two distinct types---divisors and configurations---creates the need in our repository for many lemmas transferring attributes of configurations to attributes of divisors and vice versa. In hindsight, it is possible that in the formalization this parallel terminology adds more complexity than it prevents, and the overall proof could be made shorter by removing the ``configuration'' nomenclature and simply working with divisors throughout. This issue is endemic to proof formalization and takes some acclimation; the two words make arguments simpler and easier to follow in prose because one can fluently switch between them as the number of chips at $q$ becomes relevant or irrelevant, but the situation is quite different in Lean. Nevertheless, we have preserved this dual nomenclature in an effort to faithfully mirror the organization of \cite{corry2018divisors}, which uses the vocabulary of superstable configurations throughout.
\end{remark}

\subsection{The Baker--Norine rank}

The notion of \emph{winnability} may be understood as arising from a solitaire game: given a chip configuration $D$, Alice plays alone, making chip-firing moves at will, and wins if she can eliminate debt. Subtler information about a divisor $D$ is revealed by adding a second player. Suppose now that Alice has an opponent, Bob, who is allowed to remove a chip, possibly creating debt, every time Alice succeeds in eliminating debt. Of course, Bob will always eventually cause Alice to lose, but an interesting quantitative question remains: what is the maximum number of moves Bob can make such that Alice, playing optimally, can still win? The answer to this question may be understood as quantifying ``how winnable'' $D$ is, and it is called the \emph{Baker--Norine rank} of $D$, denoted $r(D)$. By convention, an unwinnable divisor has rank $-1$. More formally, $r(D) = \deg E - 1$, where $E$ is a minimum-degree effective divisor for which $D-E$ is unwinnable.
Computing the rank on a general graph is NP-hard~\cite{kiss2015chip}, although it is a finite computation.

Defining the function $r(\cdot)$ in Lean is a useful illustration of a subtlety in how Lean handles definitions. That is, the function $r(D)$ is not defined by a formula or procedure, but by the following two propositions characterizing it.

\begin{lstlisting}
def rank_geq (G : CFGraph) (D : CFDiv G) (k : ℤ) : Prop :=
  ∀ E ∈ eff_of_degree G k, winnable G (D-E)

def rank_eq (G : CFGraph) (D : CFDiv G) (r : ℤ) : Prop :=
  rank_geq G D r ∧ ¬(rank_geq G D (r+1))
\end{lstlisting}

But Lean is a programming language, so it wants a functional definition.
This is done by first proving a lemma \docref{Rank}{rank_exists} and then preposterously defining \docref{Rank}{rank} via the axiom of choice, represented in Lean with the function
\href{https://leanprover-community.github.io/mathlib4_docs/Init/Classical.html#Classical.choose}{\texttt{\detokenize{Classical.choose}}\,\extlinkicon}.
This function takes a proof of an existential proposition, and ``returns'' a witness. Of course, this is not a ``real'' function, since existential propositions can be proved nonconstructively. The result, unfortunately, is that our \emph{definition} of $r(D)$ is \myinline|noncomputable|.

\begin{remark}\label{rem:computable}
The constructivist reader may well feel uneasy, perhaps righteously angry, that we apparently used the axiom of choice to define the result of a finite computation. We sympathize with this unease. We would be happy to see this definition replaced with a computable one in a future update. In principle, there is no obstacle: one can enumerate all effective divisors $E$ on $G$ in order of degree, checking the winnability of each $D-E$ (which is decidable, for example by Dhar's algorithm), until a non-winnable $D-E$ is found, and return $\deg E - 1$. This would have the benefit that Lean could then actually evaluate the function on small examples. We have not done so purely as a matter of convenience, and because the practical uses of a constructive definition are limited by the NP-hardness of the problem.
\end{remark}

A practical consequence of this type of definition is that one typically does not interact directly with the definition, but rather accesses its properties using interface lemmas such as \docref{Rank}{rank_geq_iff}, which equates \myinline|rank_geq G D k| to \myinline|rank G D ≥ k|.

%% ============================================================
\section{The Riemann--Roch theorem}
\label{sec:rr}
%% ============================================================

In the interest of motivating the technical matters to follow, we will now work backwards from the end goal, which is the graph-theoretic Riemann--Roch theorem, due to Baker and Norine~\cite{baker2007riemann}.

\begin{theorem}[\docref{RiemannRoch}{riemann_roch_for_graphs}]
For any divisor $D$ on a connected graph $G$ of genus $g$ with canonical divisor $K$,
\[
  r(D) - r(K - D) = \deg(D) + 1 - g.
\]
\end{theorem}

The repository also formalizes several standard corollaries of Riemann--Roch, following \cite{corry2018divisors}: Clifford's theorem (\docref{RiemannRoch}{clifford_theorem}), the determination of $r(D)$ by $\deg(D)$ whenever $\deg(D)$ lies outside the range $[0, 2g-2]$ (\docref{RiemannRoch}{rank_nonspecial_range}), the fact that $D$ is maximal unwinnable if and only if $K - D$ is (\docref{RiemannRoch}{maximal_unwinnable_symmetry}), and the bound $g+1$ on the gonality of a connected graph (\docref{RiemannRoch}{gonality_leq_genus_add_one}; see Section~\ref{sec:future} for the definition of gonality).

Our formalization, following the strategy of \cite{corry2018divisors}, first proves an inequality of the following form.

\begin{lemma}[\docref{RRGHelpers}{rank_degree_inequality}] \label{lem:RRineq}
If $D$ is a divisor of degree $d$ on a connected graph $G$, then
\[ r(K-D) < g - d + r(D). \]
\end{lemma}

The proof of Riemann--Roch from this inequality is short and clean. Apply the inequality to both $D$ and $K-D$, thereby giving an upper and lower bound on $r(D) - r(K-D)$. All that is needed in the proof is this pair of inequalities, the integrality of the rank function, and the equation $\deg K = 2g-2$; in the formalization, the \myinline|linarith| tactic handles the algebra.

We now arrive at the core conceptual link in the Riemann--Roch story, which answers: where does the duality come from? Following terminology of Mikhalkin and Zharkov~\cite{mikhalkinZharkov}, we can understand this central duality as the following symmetry property of maximal unwinnable divisors.

\begin{lemma} \label{lem:moderators}
There exists a set $\mathcal{M}$ of divisors on a connected graph $G$, called \emph{moderators}, such that:
\begin{enumerate}
\item (\docref{RRGHelpers}{moderator_symmetry}) For any moderator $M$, the divisor $K-M$ is also a moderator;
\item (\docref{RRGHelpers}{moderator_degree}) Any moderator has degree $g-1$;
\item (\docref{RRGHelpers}{unwinnable_of_moderator}) Any moderator is unwinnable;
\item (\docref{RRGHelpers}{moderator_of_unwinnable}) For any unwinnable divisor $D$, there exists an effective divisor $E$ and a moderator $M$ such that $D + E \sim M$.
\end{enumerate}
\end{lemma}

Together, these properties imply that every maximal unwinnable divisor is linearly equivalent to a moderator, and conversely that every moderator is maximal unwinnable.
The symmetry $K - \mathcal{M} = \mathcal{M}$ is the combinatorial duality that drives the Riemann--Roch formula; as we will see shortly this duality arises from the reversal of acyclic orientations on $G$. We will explain the proof and formalization of Lemma~\ref{lem:moderators} soon, but first we explain how it implies Lemma~\ref{lem:RRineq} and therefore the Riemann--Roch formula.

\begin{proof}[Proof sketch of Lemma~\ref{lem:RRineq} from Lemma~\ref{lem:moderators}]
By definition of $r(D)$, there exists an effective divisor $E$ of degree $r(D)+1$ such that $D-E$ is unwinnable. Therefore there exists a second effective divisor $F$ such that $D - E + F \sim M$ for some moderator $M$. By the symmetry of moderators, $K-M$ is also a moderator, hence unwinnable. So $(K-D) - F + E$ is unwinnable, and in turn so is $(K-D) - F$. This shows $r(K-D) < \deg F = \deg M - \deg D + \deg E = g-d+r(D)$.
\end{proof}

We now continue to work backwards from our goal, through the tools needed to establish the properties of moderators in Lemma~\ref{lem:moderators}.

%% ============================================================
\section{Existence of $q$-reduced representatives}
\label{sec:qred}
%% ============================================================

The existence and uniqueness of a \emph{$q$-reduced representative}, as defined below, for every divisor class is a crucial step in our arguments, and presents several challenges to formalization that are not immediately apparent to the reader of the one-page prose proof in, e.g., \cite[Theorem~3.6]{corry2018divisors}. This is the portion of our formalization that is most original in its approach; our goal was to give simply stated and specific induction arguments.

\begin{definition}
A divisor $D$ is \docref{Basic}{q_reduced} if it has no debt away from $q$, and set-firing any nonempty set $S$ not containing $q$ would create debt somewhere. Equivalently (by \docref{Config}{q_reduced_superstable_correspondence}), the configuration associated to $D$ is superstable.
\end{definition}

The subtlety arises from checking the termination of two recursive processes; to do so we designed an argument somewhat different from what is most natural in prose. The first step is to concentrate all debt on a single vertex. We call a divisor $D$ \docref{Basic}{q_effective} if $D(v) \ge 0$ for all $v \ne q$.

\begin{lemma}[\docref{Basic}{q_effective_exists}]
If $D$ is a divisor on a connected graph $G$, and $q \in V(G)$ is any vertex, then there exists a $q$-effective divisor $D'$ linearly equivalent to $D$.
\end{lemma}

\begin{proofsketch}
Call a subset $S \subseteq V(G)$ \docref{Basic}{benevolent} if for every divisor $D$ there exists a divisor $D' \sim D$ such that $D'(v) \ge 0$ for all $v \notin S$, i.e. if it is always possible to concentrate all debt on $S$. It suffices to prove that every nonempty set $S$ is benevolent (\docref{Basic}{benevolent_of_nonempty}). Proceed by induction on $|V(G) \setminus S|$. The base case $S = V(G)$ is trivial. Otherwise, identify a vertex $v \not\in S$ with a neighbor $w \in S$ (using connectivity), and assume that $S \cup \{v\}$ is benevolent, so we may concentrate debt there. Firing $w$ enough times then brings $v$ out of debt as well, while creating no new debt outside $S$, so $S$ is benevolent.
\end{proofsketch}

In a functional programming language such as Lean, this inductive proof is nicely captured as a recursive function, in which the proof for $S$ makes a recursive call to the proof for $S \cup \{v\}$.
A similar strategy is employed from here to convert a $q$-effective divisor into a $q$-reduced divisor.

\begin{lemma}[\docref{Basic}{q_effective_to_q_reduced}]
  If $D$ is a $q$-effective divisor on a connected graph, then there exists a $q$-reduced divisor $E \sim D$.
\end{lemma}

\begin{proofsketch}
Call a vertex $v$ \docref{Basic}{active} if some sequence of firing moves, in which $q$ is fired the least and $v$ is fired strictly more than $q$, results in a divisor that is still $q$-effective. If $D$ has no active vertices, then it is $q$-reduced. Otherwise, use connectivity to choose an active vertex $v$ with an inactive neighbor $w$, and apply the firing moves witnessing the activity of $v$ to obtain $D' \sim D$. Define the \docref{Basic}{reduction_excess} of $D$ to be the sum of $D(v)$ over all active vertices $v$. Inactive vertices never lose chips in this process, the inactive neighbor $w$ strictly gains chips, and no inactive vertex becomes active; therefore $D'$ has strictly smaller reduction excess than $D$, and the proof may recurse on $D'$.
\end{proofsketch}

\begin{corollary}[\docref{Basic}{exists_q_reduced_representative}, \docref{RRGHelpers}{superstable_of_divisor}]
For every divisor $D$ and vertex $q$ in a connected graph, there exists a $q$-reduced divisor $E$ linearly equivalent to $D$. Equivalently, there exists a superstable configuration $c$ and integer $k$ such that $D \sim c - kq$.
\end{corollary}

In fact, this representative is unique (\docref{Basic}{unique_q_reduced}). This is an important fact in chip-firing theory, but we omit the proof sketch here since it not needed to prove Riemann--Roch.

Although our proof is phrased cosmetically as an algorithm, this is misleading: we have defined the set of active vertices abstractly, without an algorithm to find it. This is sufficient for a formal proof, but in applications one may want an algorithm. One important technique for part of this process, which also pays theoretical dividends, will be described next.

%% ============================================================
\section{Dhar's burning algorithm}
\label{sec:dhar}
%% ============================================================

A core algorithmic tool in chip-firing is Dhar's burning algorithm~\cite{dhar1990self}, which plays a role in two crucial links of the proof of Riemann--Roch. Formalizing the necessary aspects of Dhar's algorithm led to some subtly difficult design choices. The algorithm is quite natural to describe in English. Its input is a configuration $c$ with respect to a marked vertex $q$, in which there is no debt away from $q$ (as we have remarked, this last condition is assumed for all configurations in our Lean definition). Now we imagine that the vertex $q$ catches on fire. The fire spreads nondeterministically as follows: at each step, an edge may catch on fire if one of its endpoints has burned, or a vertex with $n$ chips may catch on fire if at least $n+1$ of its incident edges have caught on fire. The algorithm stops once no further steps are possible.

As we have described it, Dhar's algorithm produces the following data:
\begin{enumerate}
\item A set of ``burned vertices'' $B$;
\item An ordering on the set of burned vertices, namely the order in which they caught on fire;
\item An orientation on each burned edge, pointing out of the vertex that caught it on fire.
\end{enumerate}

These data are used in two distinct ways, depending on how much of the graph burns. If some of the graph remains unburned, then the set $B$ is important: its complement is a ``legal firing set,'' meaning that if we now fire all of those vertices (ironically, given that they had just escaped the ``fire''), no additional debt is created. Therefore the burning algorithm provides a practical way to move a $q$-effective divisor closer to being $q$-reduced, or equivalently a configuration closer to being superstable.
On the other hand, if the entire graph \emph{does} burn, then it is the path of the fire, that is, the acyclic orientation of edges, that is important; it will be used to construct a moderator in the next section.

In our work, we found that the data structure most suitable for the proofs we needed focuses on the second datum listed above: the order in which the vertices burned. In fact, the implementation becomes even cleaner if we consider this list in \emph{reverse} order, because doing so is quite natural in functional programming.

\begin{definition}[\docref{Config}{burn_list}]
A \emph{burn list} for a configuration $c$ with respect to a vertex $q$ is a list $v_1, v_2, \dots, v_n, q$ of distinct vertices such that, for each $i \in \{1, \ldots, n\}$, the number of edges from $v_i$ to $\{v_{i+1}, \ldots, v_n, q\}$ is greater than $c(v_i)$. In other words, $v_i$ can catch on fire next if the vertices $v_{i+1}, \ldots, v_n, q$ have already caught on fire.
\end{definition}

If the configuration $c$ is superstable, and $v_1, v_2, \dots, v_n, q$ is a burn list, then the set $V(G) \setminus \{v_1, v_2, \dots, v_n, q\}$ is not a legal firing set unless it is empty. This means that there exists some $v_0$ for which prepending $v_0$ gives a burn list one vertex longer. That is, we can identify the next vertex $v_0$ to burn. Iterating this process gives the following lemma, which encapsulates Dhar's burning algorithm as it is used in our formalization.

\begin{lemma}[\docref{Config}{superstable_burn_list}]
If $c$ is a superstable configuration, then there exists a burn list $v_1, v_2, \dots, v_n, q$ that includes every vertex of $G$.
\end{lemma}

\begin{remark}
Although Dhar's algorithm is very much guiding the formalization, our proof of Riemann--Roch does not use a true implementation of Dhar's algorithm that one would use in practical computations; it only refers to the \emph{existence} of a burn list that might have resulted from this algorithm. An experimental implementation of Dhar's algorithm, in its more traditional form, can be found in \myinline|Algorithms.lean|.
\end{remark}

%% ============================================================
\section{Orientations and moderators}
\label{sec:orient}
%% ============================================================

To construct the moderators providing the essential symmetry in Lemma~\ref{lem:moderators}, we exploit a remarkable bijection between the acyclic orientations of the graph with unique source $q$ and its maximal superstable configurations with respect to $q$.

An orientation (\docref{Orientation}{CFOrientation}) of $G$ assigns a direction to each edge. We formalize this as a multiset of ordered pairs satisfying two properties: it is count-preserving, and it has no bidirectional edges.
An orientation is \emph{acyclic} if every directed path is non-repeating. The orientation divisor $D(\mathcal{O})$ assigns $\indeg_{\mathcal{O}}(v) - 1$ to each vertex $v$. These are the \emph{moderators} promised in Lemma~\ref{lem:moderators}: we define (\docref{RRGHelpers}{is_moderator})
\[ \mathcal{M} = \{D(\mathcal{O}) : \mathcal{O} \text{ is an acyclic orientation}\}.\]

\begin{remark}
The moderators are precisely the divisors $\nu_P$ used in Baker and Norine's original proof~\cite{baker2007riemann}: for a total order $<_P$ on $V(G)$, their divisor $\nu_P$, which assigns $\#\{e = vw : w <_P v\} - 1$ chips to each vertex $v$, equals $D(\mathcal{O}_P)$ for the acyclic orientation $\mathcal{O}_P$ directing each edge toward its $<_P$-larger endpoint, and every acyclic orientation arises in this way from any of its topological orders. Under this dictionary, reversing an orientation corresponds to reversing the total order, which is how the duality appears in \cite{baker2007riemann}. A systematic development of Riemann--Roch theory from the perspective of orientations may be found in \cite{backman2017orientations}.
\end{remark}

\begin{proof}[Proof sketch of Lemma~\ref{lem:moderators}]
Denote by $\overline{\mathcal{O}}$ the reverse of an orientation $\mathcal{O}$. Then it follows directly from definitions that $D(\mathcal{O}) + D(\overline{\mathcal{O}}) = K$, which gives the symmetry $\mathcal{M} = K - \mathcal{M}$. The degree of any moderator is $|E(G)|-|V(G)| = g-1$, since $\sum_{v \in V(G)} \indeg_{\mathcal{O}}(v) = |E(G)|$. This last calculation, while easily glossed over in prose, requires some care in the formalization; our strategy is to reinterpret it as a sum over $V(G) \times E(G)$ and change the order of summation; this is encapsulated in the reusable lemma \docref{Basic}{sum_card_filter_eq_mul}, which can be viewed as a generalized handshaking lemma for multigraphs.

The fact that moderators are unwinnable is proved in \docref{Orientation}{ordiv_unwinnable}; we only sketch it here. Suppose that $D(\mathcal{O})$ is winnable. Let $S$ be the set of vertices that are fired the most in a winning sequence of firing moves for $D(\mathcal{O})$. In the formalization, the firing sequence is represented by a function $\sigma : V \to \Z$, and $S$ is the set where $\sigma$ attains its maximum value. Since $\mathcal{O}$ is acyclic, we can find at least one ``relative source'' for $\mathcal{O}$ in $S$, meaning a vertex $v$ with no directed edges to it from another vertex in $S$. Then the firing moves send at least one chip away from $v$ along each incoming edge; since $v$ begins with only $\indeg_{\mathcal{O}}(v) - 1$ chips, it ends in debt, a contradiction. So $D(\mathcal{O})$ is unwinnable.

Finally, \docref{RRGHelpers}{moderator_of_unwinnable} asserts that, up to linear equivalence, every unwinnable divisor is dominated by a moderator. The formalized proof proceeds in three steps. First, we may choose a source vertex $q$ and assume $D$ is $q$-reduced. Since it is unwinnable, it has the form $D = c - kq$, where $c$ is a superstable configuration (\docref{RRGHelpers}{superstable_of_divisor}) and $k > 0$ (\docref{RRGHelpers}{superstable_of_divisor_negative_k}). Next, we may add an effective divisor $E$ to $c$ to obtain a \emph{maximal} superstable configuration $c'$, with \docref{Orientation}{maximal_superstable_exists}. Finally, we may sort the vertices $V(G)$ into a burn list $v_1, \ldots, v_n, q$ for $c'$, and from this define an orientation $\mathcal{O}$ as follows: every edge $(u,v)$ is oriented from the later to the earlier element in the burn list (in other words, it is oriented in the direction that the fire spreads); this is encapsulated in \docref{Orientation}{maximal_superstable_orientation}. The definition of burn list guarantees that $c' - q \le D(\mathcal{O})$, and the maximality of $c'$ forces equality: $D(\mathcal{O}) = c' - q$. Setting $F = (c' - c) + (k-1)q$, which is effective since $k > 0$, we conclude that $D + F \sim D(\mathcal{O})$, as desired.
\end{proof}

This completes the proof of Riemann--Roch for graphs. Although it is not needed for Riemann--Roch, our repository proves a bit more, following \cite{corry2018divisors}: not only does every maximal superstable configuration sit below a moderator, there is in fact a bijection between maximal superstable configurations with respect to $q$ and acyclic orientations with unique source $q$; this is proved as \docref{Orientation}{orientation_superstable_bijection}.

%% ============================================================
\section{Future work and formalization challenges}
\label{sec:future}
%% ============================================================

This formalization opens up several avenues for future development. One is graph-theoretic Brill--Noether theory. The tropical proof of Brill--Noether~\cite{cools2012tropical} relies on chip-firing on chains of loops, the combinatorics of which may be fertile ground for building on our repository. Moreover, the still-open Brill--Noether conjecture for graphs (stated below) is a tantalizing longer-term challenge that would require both theoretical progress and formalization work. With this in mind, we have included formalized \emph{statements} in our repository of the following four claims, in the hopes that attempts to formalize them, in whole or part, will provide a challenge and benchmark for proof formalization. Two of these conjectures are based on the following terminology.

\begin{definition}
The divisorial \docref{RiemannRoch}{gonality} of a connected graph $G$ is the minimum degree $d$ of a divisor $D$ such that $r(D) \ge 1$.
\end{definition}

The four challenges are as follows.

\begin{enumerate}
\item \docref{RiemannRoch}{max_gonality_existence} (known, not yet formalized): for every $g \ge 0$, there exists a connected graph of genus $g$ with gonality exactly $\lceil \frac12 g \rceil + 1$.

This was conjectured in \cite[Conjecture 3.10(2)]{baker2008}, and proved in \cite{cools2012tropical} using chains of loops, and again by a different construction in \cite{hendrey2018}. In fact, for $g \ge 1$ every gonality $k$ between $2$ and $\lceil \frac12 g \rceil + 1$ is possible, for example using a chain of loops with torsion orders $k$; see e.g. \cite{pflueger2017brill}.

\item \docref{RiemannRoch}{brill_noether_general_existence} (known, not yet formalized): for every $g \ge 0$, there exists a connected graph of genus $g$ that is \emph{Brill--Noether general} in the following sense: if the graph has a divisor $D$ of rank $r$ and degree $d$, then the Brill--Noether number $\rho = g - (r+1)(g-d+r)$ satisfies $\rho \ge 0$. Our formalized statement states this in following form, equivalent by Riemann--Roch: for every divisor $D$ on $G$, $(r(D)+1)\,(r(K-D)+1) \le g$.

A slightly different form of this was conjectured in \cite[Conjecture 3.9(2)]{baker2008}, and it was proved using chains of loops in \cite{cools2012tropical}. Note that this implies that the gonality of $G$ is at least $\lceil \frac12 g \rceil + 1$, since $r = 1$ implies $\rho = -g - 2 + 2d$, so a slightly strengthened form of this statement will also solve the first challenge.

\item \docref{RiemannRoch}{gonality_conjecture} (open problem): the gonality of every connected graph of genus $g$ is at most $\lceil \frac12 g \rceil + 1$.

This is conjectured in \cite[Conjecture 3.10(1)]{baker2008}. It is known for \emph{metric} graphs; this was first proved via algebraic geometry \cite[Theorem 3.12, $r=1$ case]{baker2008} and by pure combinatorics in \cite[Theorem 2]{draismaVargas21}.

\item \docref{RiemannRoch}{brill_noether_conjecture} (open problem): for every connected graph $G$ of genus $g$ and all positive integers $d,r$ satisfying $\rho = g - (r+1)(g-d+r) \ge 0$, the set $W^r_d(G)$ of degree-$d$ divisor classes with rank $r(D) \ge r$ is nonempty.

This is conjectured in \cite[Conjecture 3.9(1)]{baker2008}. The case $r=1$ is the gonality conjecture. It is known for \emph{metric} graphs \cite[Theorem 3.12]{baker2008}, but the only known argument requires specialization from algebraic curves. It remains open to find a combinatorial proof even of the metric graph version.
\end{enumerate}

A nice introduction to graph gonality can be found in \cite{beougher2023chip} and the references therein, where the reader can find many other theorems and conjectures that may be excellent formalization projects.

Another promising direction for future formalization work is in combinatorial \emph{Hurwitz--Brill--Noether theory}, which is the study of the sets $W^r_d(G)$ defined above when $G$ is assumed to have a specific gonality $k$. We have not yet formalized any statements in this direction, but the reader may find many interesting candidates in e.g. \cite{pflueger2017brill, jensenranganathan2021fixedgonality, cookpowelljensen2022components, cookpowelljensen2022hurwitz, jensenSawyer}.

In quite another direction, a longer-term goal is formalizing Baker's specialization lemma~\cite{baker2008}, connecting chip-firing games to algebraic geometry via degeneration of algebraic curves. Such an effort would require extensive work in formalized algebraic geometry.

Finally, we express our hope that chip-firing games, and the work in our repository, can eventually be integrated into the \Mathlib library. Such work will take time and care, with many important design decisions, such as the proper level of generality for the core definitions. In our opinion, the first major milestone is probably to incorporate the proof of existence for $q$-reduced representatives, or equivalently superstable configurations.
The authors would welcome collaborators who wish to work towards this goal.

\bibliographystyle{plain}
\bibliography{references}

\end{document}